\documentclass[12pt]{amsart}
\usepackage[utf8]{inputenc}
\usepackage[english]{babel}
\usepackage{amssymb, mathtools, mathdots, amsthm, amsmath}
\usepackage{tikz}

\usepackage{comment}
\usepackage{float, caption, tabu}
\usepackage[autostyle, english = american]{csquotes}
\MakeOuterQuote{"}

\usepackage[inner=1.0in,outer=1.0in,bottom=1.0in, top=1.0in]{geometry}
\usepackage{hyperref}

\newtheorem{theorem}{Theorem}[section]
\newtheorem*{theorem*}{Theorem}
\newtheorem{proposition}[theorem]{Proposition}
\newtheorem*{proposition*}{Proposition}
\newtheorem{lemma}[theorem]{Lemma}
\newtheorem*{lemma*}{Lemma}
\newtheorem{corollary}[theorem]{Corollary}
\newtheorem*{corollary*}{Corollary}
\newtheorem{conjecture}{Conjecture}
\newtheorem*{conjecture*}{Conjecture}
\numberwithin{equation}{section}

\renewcommand{\H}{\mathbb{H}}

\newcommand{\Z}{\mathbb{Z}}

\newcommand{\pmat}[1]{\begin{pmatrix} #1 \end{pmatrix}}

\newcommand{\abs}[1]{\left\vert #1 \right\vert}

\DeclareMathOperator{\maxN}{maxN}

\title[Bounds for Coefficients of the $f(q)$ Mock Theta Function]{Bounds for Coefficients of the $f(q)$ Mock Theta Function and Applications to Partition Ranks}
\author{Kevin Gomez}
\address{Vanderbilt University\\
  Nashville, Tennessee 37235}
\email{kevin.j.gomez@vanderbilt.edu}

\author{Eric Zhu}
\address{Georgia Institute of Technology\\
  Atlanta, Georgia 30332}
\email{ezhu31@gatech.edu}

\begin{document}

\maketitle

\begin{abstract}
    We compute effective bounds for $\alpha(n)$, the Fourier coefficients of Ramanujan's mock theta function $f(q)$ utilizing a finite algebraic formula due to Bruinier and Schwagenscheidt. We then use these bounds to prove two conjectures of Hou and Jagadeesan on the convexity and maximal multiplicative properties of the even and odd partition rank counting functions.
\end{abstract}

\section{Introduction and Statement of Results} \label{sec:intro}
For a nonnegative integer $n$, a partition of $n$ is a finite list of nondecreasing positive integers $\lambda = (\lambda_1, \lambda_2, \dots, \lambda_k)$ such that $\lambda_1 + \lambda_2 + \cdots + \lambda_k = n$. The partition number $p(n)$ denotes the number of partitions of $n$ which has been of large interest to number theorists. 

Given a partition $\lambda$ of $n$, the rank of $\lambda$ is defined as $\lambda_k - k$. In words, this is the largest part of the partition minus the number of parts. For any $n$, we can consider $N(r,t;n)$ which counts the number of partitions of $n$ that have rank equal to $r \pmod{t}$.

For the case of $t = 2$, we analyze partitions with even or odd rank, captured by the coefficients $\alpha(n)$ of Ramanujan's mock theta function
\begin{equation*}
    \begin{split}
        f(q) &:= 1 + \sum_{n=0}^{\infty} \frac{q^{n^2}}{(1 + q)^2(1 + q^2)^2 \dots (1 + q^n)^2} \\
        &= 1 + \sum_{n=0}^{\infty} \alpha(n)q^n
    \end{split}
\end{equation*}
for $q := e^{2\pi i z}$, where $\alpha(n) = N(0,2;n) - N(1,2;n)$.

In this paper, we will prove the following asymptotic formula for $\alpha(n)$ with an effective bound on the error term:
\begin{theorem} \label{thm:main}
    Let $D_n := -24n+1$ and $l(n) := \pi\sqrt{\abs{D_n}}/6$. Then for all $n \geq 1$,
    \begin{equation*}
        \alpha(n) = (-1)^{n+1} \frac{\sqrt{6}}{\sqrt{24n-1}} e^{l(n)/2} + E(n)
    \end{equation*}
    where
    \begin{equation*}
        \abs{E(n)} < (4.30 \times 10^{23}) 2^{q(n)} \abs{D_n}^2 e^{l(n)/3}
    \end{equation*}
    with
    \begin{equation*}
        q(n) := \frac{\log(\abs{D_n})}{\abs{\log \log(\abs{D_n}) - 1.1714}}.
    \end{equation*}
\end{theorem}

In 1966, Andrews and Dragonette \cite[pp. 456]{dragonette} conjectured a Rademacher-type infinite series for $\alpha(n)$. This conjecture was proved by Bringmann and Ono \cite{ono}, who obtained the following formula:
\begin{equation}\label{BO}
    \alpha(n) = \pi (24n - 1)^{-\frac14} \sum_{k = 1}^\infty \frac{(-1)^{\lfloor \frac{k+1}{2} \rfloor} A_{2k}\left(n - \frac{k(1 + (-1)^k)}{4}\right)}{k}  \cdot I_{1/2}\left( \frac{\pi\sqrt{24n - 1}}{12k} \right)
\end{equation}
where $A_{2k}(n)$ is a certain twisted Kloosterman-type sum and $I_{1/2}$ is the $I$-Bessel function of order 1/2. One can easily show that the $k=1$ term in (\ref{BO}) agrees with the main term in Theorem \ref{thm:main}. Since (\ref{BO}) is only conditionally convergent, it is difficult to bound. Using a different, finite algebraic formula for $\alpha(n)$ due to Alfes \cite{alfes}, Masri \cite[Theorem 1.3]{equi2} gave an asymptotic formula for $\alpha(n)$ with a power-saving error term. Ahlgren and Dunn \cite[Theorem 1.3]{dunn} also produced an asymptotic formula with a power-saving error term by bounding the series (\ref{BO}) directly.

Using Theorem \ref{thm:main}, we will show a certain convexity property for $N(r, 2; n)$. In particular, we aim to prove the following conjecture of Hou and Jagadeesan \cite[Conjecture 4.1]{hou}:
\begin{conjecture}[Hou/Jagadeesan] \label{conj:main}
    If $r = 0$ (resp. $r = 1$), then we have that
    \begin{equation*}
        N(r,2;a)N(r,2;b) > N(r,2;a+b)
    \end{equation*}
    for all $a,b \geq 11$ (resp. $12$).
\end{conjecture}
Hou and Jagadeesan \cite[Theorem 1.1]{hou} proved a similar convexity bound modulo $3$; however, their techniques do not extend to modulus two. Here, we overcome these difficulties using Theorem \ref{thm:main} and prove the following:
\begin{theorem} \label{thm:conj}
    Conjecture \ref{conj:main} is true.
\end{theorem}

Hou and Jagadeesan also discussed a direct consequence of Conjecture \ref{conj:main}, analogous to their own result for partition ranks modulo 3 \cite[Theorem 1.2]{hou}. Extend $N(r,t;n)$ to partitions as in \cite{ono} by
\begin{equation*}
    N(r,t;\lambda) := \prod_{j=1}^{k} N(r,t;\lambda_j).
\end{equation*}

Let $P(n)$ denote the set of all partitions of $n$. Hou and Jagadeesan conjectured \cite[Conjecture 4.2]{hou} the maximal values of these functions over $P(n)$ for $t = 2$, where the maximal value is defined as
\begin{equation*}
    \maxN(r,t;n) := \max(N(r,t;\lambda) \colon \lambda \in P(n)),
\end{equation*}
and characterized the partitions which attain them.
\begin{conjecture}[Hou/Jagadeesan] \label{conj:maxN}
    The following are true:
    \begin{enumerate}
        \item If $n \geq 5$, then we have that
        \begin{equation*}
            \maxN(0,2;n) =
            \begin{cases}
                3^{\frac{n}{3}} & n \equiv 0 \pmod{3} \\
                11 \cdot 3^{\frac{n-7}{3}} & n \equiv 1 \pmod{3} \\
                5 \cdot 3^{\frac{n-5}{3}} & n \equiv 2 \pmod{3},
            \end{cases}
        \end{equation*}
        and it is achieved at the unique partitions
        \begin{align*}
            &(3,3,\dots,3) \text{ when } n \equiv 0 \pmod{3} \\
            &(3,3,\dots,7) \text{ when } n \equiv 1 \pmod{3} \\
            &(3,3,\dots,5) \text{ when } n \equiv 2 \pmod{3}.
        \end{align*}
        
    \item If $n \geq 8$, then we have that
    \begin{equation*}
        \maxN(1,2;n) =
        \begin{cases}
            2^{\frac{n}{2}} & n \equiv 0 \pmod{2} \\
            12 \cdot 2^{\frac{n-9}{2}} & n \equiv 1 \pmod{2},
        \end{cases}
    \end{equation*}
    and it is achieved at the following classes of partitions
    \begin{align*}
        &(2,2,\dots,2) \text{ when } n \equiv 0 \pmod{2} \\
        &(2,2,\dots,9) \text{ when } n \equiv 1 \pmod{2}
    \end{align*}
    up to any number of the following substitutions: $(2,2) \to (4)$ and $(2,2,2) \to (6)$.
    \end{enumerate}
\end{conjecture}

\textbf{Remark.} Conjecture \ref{conj:maxN}, part (1) is a slight refinement of Hou and Jagadeesan's original claim. The result holds for $n \geq 5$ rather than $n \geq 6$.

\vspace{0.1in}

Utilizing Theorem \ref{thm:main}, we prove the following:
\begin{theorem} \label{thm:maxN}
    Conjecture \ref{conj:maxN} is true.
\end{theorem}

We also demonstrate effective equidistribution of partition ranks modulo 2 (see Corollary \ref{cor:equi}). Asymptotic equidistribution of partition ranks modulo $t$ was demonstrated by Males \cite{equi} for all $t \geq 2$. Masri \cite{equi2} proved equidistribution of partition ranks modulo 2 with a power-saving error term, however his results were not effective, and so cannot be applied toward Conjecture \ref{conj:main}.

To give an effective bound on the error term for $\alpha(n)$, we
will use a finite algebraic formula due to Bruinier and Schwagenscheidt \cite[Theorem 3.1]{BruSch} which express $\alpha(n)$ as a trace over singular moduli. To state this formula, consider the weight zero weakly-holomorphic modular form for $\Gamma_0(6)$ defined by
\begin{equation} \label{eq:F-fourier}
    F(z) := -\frac{1}{40} \frac{E_4(z) + 4E_4(2z) - 9E_4(3z) - 36E_4(6z)}{(\eta(z) \eta(2z) \eta(3z) \eta(6z))^2} = q^{-1} - 4 - 83q - 296q^2 + \dots .
\end{equation}
Bruinier and Schwagenscheidt \cite[Theorem 3.1]{BruSch} proved
\begin{theorem*}[Bruinier/Schwagenscheidt]
    For $n \geq 1$, we have
    \begin{equation*}
        \alpha(n) = -\dfrac{1}{\sqrt{\abs{D_n}}} \mathrm{Im}(S(n))
    \end{equation*}
    where
    \begin{equation*}
        S(n) := \sum_{[Q]} F(\tau_Q).
    \end{equation*} 
    Here, the sum is over the $\Gamma_0(6)$ equivalence classes of discriminant $D_n$ positive definite, integral binary quadratic forms $Q = [a, b, c]$ such that $6 \mid a$ and $b \equiv 1 \pmod{12}$, and $\tau_Q$ is the Heegner point given by the root $Q(\tau_Q,1)$ in the complex upper half-plane $\H$.
\end{theorem*}

Our proof of Theorem \ref{thm:main} is inspired by work of Locus-Dawsey and Masri \cite{spt}, who used a similar finite algebraic formula due to Ahlgren and Andersen \cite{spt-trace} for the Andrews smallest-parts function $\mathrm{spt}(n)$ to give an asymptotic formula for $\mathrm{spt}(n)$ with an effective bound on the error term and prove several conjectural inequalities of Chen \cite{chen}.

\vspace{0.1in}

\textbf{Organization}. The paper is organized as follows. In Section \ref{sec:quad}, we review some facts regarding quadratic forms and Heegner points. In Section \ref{sec:F}, we derive the Fourier expansion of $F(z)$ and effective bounds on its coefficients. In Section \ref{sec:main}, we prove Theorem \ref{thm:main}. In Section \ref{sec:cor}, we discuss corollaries to Theorem \ref{thm:main}. In Section \ref{sec:conj}, we prove Theorem \ref{thm:conj}. Finally, in Section \ref{sec:maxN}, we prove Theorem \ref{thm:maxN}.

\vspace{0.1in}

 \textbf{Acknowledgements}. We would like to thank Riad Masri, Matthew Young, and Agniva Dasgupta for their support in this work. We especially thank Narissara Khaochim for her contributions to the proof of Proposition \ref{prop:F-fourier} and Andrew Lin for very helpful comments. We also thank the referees for their detailed suggestions to improve the exposition. This research was completed in the 2020 REU in the Department of Mathematics at Texas A\&M University, supported by NSF grant DMS-1757872.

\section{Quadratic Forms and Heegner Points} \label{sec:quad}
Let $N$ be a positive integer and $D$ be a negative integer discriminant coprime to $N$. Let $\mathcal{Q}_{D,N}$ be the set of positive definite, integral binary quadratic forms \begin{equation*}
    Q(X,Y) = [a, b, c](X,Y) = aX^2 + bXY + c Y^2
\end{equation*}
with discriminant $b^2 - 4ac = D < 0$ with $a \equiv 0 \pmod{N}$. The congruence subgroup $\Gamma_0(N)$ acts on $\mathcal{Q}_{D,N}$ by \begin{equation*}
    Q \circ \sigma = [a^\sigma, b^\sigma, c^\sigma]
\end{equation*}
with $\sigma = \begin{pmatrix}
w & x \\
y & z
\end{pmatrix} \in \Gamma_0(N)$, where
\begin{align*}
    a^\sigma &= aw^2 + b wy + c y^2 \\
    b^\sigma &= 2a wx + b(wz+ x y) + 2 cyz \\
    c^\sigma &= a x^2 + bx z + c z^2.
\end{align*}

Given a solution $r \pmod{2N}$ of $r^2 \equiv D \pmod{4N}$, we define the subset of forms
\begin{equation*}
    \mathcal{Q}_{D,N,r} := \{Q = [a,b,c] \in \mathcal{Q}_{D, N} \mid b \equiv r\pmod{2N} \}.
\end{equation*}
The group $\Gamma_0(N)$ also acts on $\mathcal{Q}_{D,N,r}$. The number of $\Gamma_0(N)$ equivalence classes in $\mathcal{Q}_{D,N,r}$ is given by the Hurwitz-Kronecker class number $H(D)$.

We can also consider the subset $\mathcal{Q}^{\text{prim}}_{D,N}$ of primitive quadratic forms in $\mathcal{Q}_{D,N}$. These are the forms such that
\begin{equation*}
    \gcd(a,b,c) = 1.
\end{equation*}
In this case, the number of $\Gamma_0(N)$ equivalence classes in $\mathcal{Q}^{\mathrm{prim}}_{D,N,r}$ is given by the class number $h(D)$. 

To each form $Q \in \mathcal{Q}_{D,N}$, we associate a Heegner point $\tau_Q$ which is the root of $Q(X,1)$ given by
\begin{equation*}
    \tau_Q = \frac{-b + \sqrt{D}}{2a} \in \H.
\end{equation*}
The Heegner points $\tau_Q$ are compatible with the action of $\Gamma_0(N)$ in the sense that if $\sigma \in \Gamma_0(N)$, then
\begin{equation} \label{eq:compat}
    \sigma(\tau_Q) = \tau_{Q \circ \sigma^{-1}}.
\end{equation}

\section{Fourier Expansion of \texorpdfstring{$F(z)$}{F(z)}} \label{sec:F}

Let $D_n = -24n+1$ for $n \in \Z^+$ and define the trace of $F(z)$ by
\begin{equation*}
    S(n) := \sum_{[Q] \in \mathcal{Q}_{D_n,6,1}/\Gamma_0(6)} F(\tau_Q).
\end{equation*}
Proceeding as in \cite[Section 3]{spt}, we decompose $S(n)$ as a linear combination involving traces of primitive forms. Let $\Delta <0$ be a discriminant with $\Delta \equiv 1 \pmod{24}$ and define the class polynomials 
\begin{equation*}
    H_\Delta(X) := \prod_{[Q] \in \mathcal{Q}_{\Delta, 6, 1}/\Gamma_0(6)} (X - F(\tau_Q))
\end{equation*}
and 
\begin{equation*}
    \widehat{H}_{\Delta, r}(F; X) := \prod_{[Q] \in \mathcal{Q}^{\text{prim}}_{\Delta, 6, r}/ \Gamma_0(6)} (X - F(\tau_Q)).
\end{equation*}

Let $\{W_\ell\}_{\ell \mid 6}$ be the group of Atkin-Lehner operators for $\Gamma_0(6)$. We have by \cite[pp. 47]{BruSch}
\begin{equation} \label{eq:F-eigen} 
    F\vert_0 W_\ell = \beta(\ell) F
\end{equation}
where $\beta(\ell) = 1$ if $\ell = 1,2$ and $\beta(\ell) = -1$ if $\ell = 3,6$.

Using these eigenvalues we modify \cite[Lemma 3.7]{epsilon} to get the following:

\begin{lemma} \label{lem:decomp}
We have the decomposition
    \begin{equation*}
        H_{\Delta}(X) = \prod_{\substack{u > 0 \\ u^2 \mid \Delta}} \varepsilon(u)^{h(\Delta/u^2)} \widehat{H}_{\Delta/u^2, 1}(F; \varepsilon(u)X)
    \end{equation*}  where $\varepsilon(u) = 1$ if $u \equiv 1, 7 \pmod{12}$ and $\varepsilon(u) = -1$ if $u \equiv 5, 11 \pmod{12}$.
\end{lemma}

Comparing coefficients on both sides of Lemma 3.1 yields the decomposition
\begin{equation} \label{eq:s-decomp}
    S(n) = \sum_{\substack{u > 0 \\ u^2 \mid D_n}} \varepsilon(u) S_u(n)
\end{equation}
where
\begin{equation*}
    S_u(n) := \sum_{[Q] \in \mathcal{Q}_{D_n/u^2,6,1}^{\mathrm{prim}}/\Gamma_0(6)} F(\tau_Q).
\end{equation*}

We now express $S_u(n)$ as a trace involving primitive forms of level 1. As in \cite[Section 3]{spt}, we let $\mathbf{C}_6$ denote the following set of right coset representatives of $\Gamma_0(6)$ in $SL_2(\Z)$:
\begin{align*}
    \gamma_{\infty} &:= \pmat{1 & 0 \\ 0 & 1} \\
    \gamma_{1/3,r} &:= \pmat{1 & 0 \\ 3 & 1} \pmat{1 & r \\ 0 & 1}, \quad r = 0, 1 \\
    \gamma_{1/2,s} &:= \pmat{1 & 1 \\ 2 & 3} \pmat{1 & s \\ 0 & 1}, \quad s = 0, 1, 2 \\
    \gamma_{0,t} &:= \pmat{0 & -1 \\ 1 & 0} \pmat{1 & t \\ 0 & 1}, \quad t = 0, 1, 2, 3, 4, 5
\end{align*}
where $\gamma_{\infty}(\infty)$, $\gamma_{1/3,r}(\infty) = 1/3$, $\gamma_{1/2,s}(\infty) = 1/2$, and $\gamma_{0,t}(\infty) = 0$.

Recall that a form $Q = [a_Q,b_Q,c_Q] \in \mathcal{Q}_{\Delta,1}$ is reduced if
\begin{equation*}
    \abs{b_Q} \leq a_Q \leq c_Q,
\end{equation*}
and if either $\abs{b_Q} = a_Q$ or $a_Q = c_Q$, then $b_Q \geq 0$. Let $\mathcal{Q}_\Delta$ denote a set of primitive, reduced forms representing the equivalence classes in $\mathcal{Q}_{\Delta,1}^{\mathrm{prim}}/ SL_2(\Z)$. For each $Q \in \mathcal{Q}_\Delta$, there is a unique choice of representative $\gamma_Q \in \mathbf{C}_6$ such that
\begin{equation*}
    [Q \circ \gamma_Q^{-1}] \in \mathcal{Q}_{\Delta,6,1}^{\mathrm{prim}}/\Gamma_0(6).
\end{equation*}
This induces a bijection
\begin{equation} \label{eq:bijection}
    \begin{split}
        \mathcal{Q}_\Delta &\longrightarrow \mathcal{Q}_{\Delta,6,1}^{\mathrm{prim}}/\Gamma_0(6) \\
        Q &\longmapsto [Q \circ \gamma_Q^{-1}];
    \end{split}
\end{equation}
see \cite[pp. 505]{decomp}, or more concretely, \cite[Lemma 3]{forms}, where an explicit list of the matrices $\gamma_Q \in \mathbf{C}_6$ is given.

Using the bijection (\ref{eq:bijection}) and the compatibility relation (\ref{eq:compat}) for Heegner points, the trace $S_u(n)$ can be expressed as
\begin{equation} \label{eq:su-decomp}
    S_u(n) = \sum_{[Q] \in \mathcal{Q}_{D_n/u^2,6,1}^{\mathrm{prim}}/\Gamma_0(6)} F(\tau_Q) = \sum_{Q \in \mathcal{Q}_{D_n/u^2}} F(\gamma_Q(\tau_Q)).
\end{equation}

Therefore, to study the asymptotic distribution of $S_u(n)$, we need the Fourier expansion of $F(z)$ with respect to $\gamma_\infty, \gamma_{1/3,r}, \gamma_{1/2,s}$, and $\gamma_{0,t}$.

We first find the Fourier expansion of $F(z)$ at the cusp $\infty$.

\begin{proposition} \label{prop:F-fourier}
    The Fourier expansion of $F(z)$ at the cusp $\infty$ is 
    \begin{align*} 
        F(z) = \sum_{n=-1}^{\infty} a(n)e(nz)
    \end{align*}
    where $a(-1)=1$, $a(0)=-4$ and for $n \geq 1$,
    \begin{align*}
        a(n) = \frac{2\pi}{\sqrt{n}} \sum_{\ell \mid6}\frac{\beta(\ell)}{\sqrt{\ell}} \sum_{\mathclap{\substack{c>0\\ c \equiv 0 \ (\text{mod} \ 6/\ell)\\ (c,\ell)=1}}} \ c^{-1} S(-\bar{\ell},n;c)I_{1}\left(\frac{4\pi \sqrt{n}}{c\sqrt{\ell}}\right),
    \end{align*}
    where
    \begin{align*}
        \beta(\ell):=
        \begin{cases}
            1, & \ell =1,2 \\
            -1, & \ell =3,6,
        \end{cases}
     \end{align*}
     $I_1$ is the $I$-Bessel function of order 1, and $S(a,b;c)$ is the ordinary Kloosterman sum defined as follows
    \begin{align*}
        S(a,b;c):= \sum_{\mathclap{\substack{d~(\mathrm{mod}~{c})\\(c,d)=1}}} ~ e\left( \frac{a\bar{d}+bd}{c} \right),
    \end{align*}
    with $\bar{d}$ the multiplicative inverse of $d \pmod{c}$.
\end{proposition}

\begin{proof}
    Define the function
    \begin{align*}
        \mathcal{P}_F(z) := 2 \sum_{\ell \mid 6} \beta(\ell)F_1(W_\ell z,1,0)
    \end{align*}
    where $F_1(z,1,0)$ is the Poincare series
    \begin{equation*}
        F_1(z,1,0) := \frac{1}{2} \sum_{\gamma \in \Gamma_{\infty}\backslash \Gamma_0(6)} [M_{0,1/2}(4\pi y)e(-x)] \, \vert_0 \, \gamma
    \end{equation*}
    for $M_{\kappa,\mu}$ the usual Whittaker function. Then by a straightforward calculation, we have
    \begin{align*}
        \mathcal{P}_F(z) := 2 \sum_{\ell \mid 6} \beta(\ell)\sum_{\gamma \in \Gamma_{\infty} \backslash \Gamma_{0}(6)} g(\gamma W_\ell z)
    \end{align*}
    where
    \begin{align*}
        g(z) := \psi(y) e(-z),
    \end{align*}
    and 
    \begin{align*}
        \psi(y) := \pi \sqrt{y}I_{1/2}(2\pi y) e^{-2\pi y}.
    \end{align*}
    Now, arguing as in \cite[Section 2]{partition}, we get the Fourier expansion
    \begin{align*}
        \mathcal{P}_F(z) = e(-z)-e(-\bar{z})+ b_F(0) + \sum_{n=1}^{\infty} b_F(-n)e(-n\bar{z}) + \sum_{n=1}^{\infty} b_F(n)e(nz),
    \end{align*}
    where
    \begin{align*}
        b_F(0) := 4\pi^2 \sum_{\ell \mid6}\frac{\beta(\ell)}{\ell} \sum_{\mathclap{\substack{c>0\\ c \equiv 0 \ (\text{mod} \ 6/\ell)\\ (c,\ell)=1}}}c^{-2} S(-\bar{\ell},0;c),
    \end{align*}
    and for $n>0$
    \begin{align*}
        b_F(-n):= \frac{2\pi}{\sqrt{n}} \sum_{\ell \mid6}\frac{\beta(\ell)}{\sqrt{\ell}} \sum_{\mathclap{\substack{c>0\\ c \equiv 0 \ (\text{mod} \ 6/\ell)\\ (c,\ell)=1}}} \ c^{-1} 
        S(-\bar{\ell},-n;c)J_{1}\left(\frac{4\pi \sqrt{n}}{c\sqrt{\ell}}\right),
    \end{align*}
    and
    \begin{align*}
        b_F(n):= \frac{2\pi}{\sqrt{n}} \sum_{\ell \mid6}\frac{\beta(\ell)}{\sqrt{\ell}} \sum_{\mathclap{\substack{c>0\\ c \equiv 0 \ (\text{mod} \ 6/\ell)\\ (c,\ell)=1}}} \ c^{-1} 
        S(-\bar{\ell},n;c)I_{1}\left(\frac{4\pi \sqrt{n}}{c\sqrt{\ell}}\right).
    \end{align*}
    
    By (\ref{eq:F-fourier}), we have $a(-1) = 1$ and $a(0) = -4$ so that
    \begin{equation*}
        F\vert_0 \gamma_\infty(z) = e(-z) - 4 + \sum_{n=1}^{\infty} a(n)e(nz).
    \end{equation*}
    The Atkin-Lehner operators for $\Gamma_0(6)$ are given by
    \begin{equation*}
        W_1 = \pmat{1 & 0 \\ 0 & 1}, \quad W_2 = \frac{1}{\sqrt{2}} \pmat{2 & -1 \\ 6 & -2}, \quad W_3 = \frac{1}{\sqrt{3}} \pmat{3 & 1 \\ 6 & 3}, \quad W_6 = \frac{1}{\sqrt{6}} \pmat{0 & -1 \\ 6 & 0}.
    \end{equation*}
    For each $\ell \mid 6$ and $v = 6/\ell$, let $V_\ell = \sqrt{\ell} W_\ell$ and
    \begin{equation*}
        A_\ell = \pmat{ \frac{1}{\text{width of the cusp } 1/v} & 0 \\ 0 & 1}.
    \end{equation*}
    
    We have
     \begin{table}[H]
        \centering
        \begin{tabu}{|c|c|c|c|c|} \hline
          cusp $1/v$ & $\infty \simeq 1/6$ & $1/3$ & $1/2$ & $0 \simeq 1$ \\ \hline
          $\ell$ & 1 & 2 & 3 & 6 \\ \hline
          $V_\ell$ & $\pmat{1 & 0 \\ 0 & 1}$ & $\pmat{2 & -1 \\ 6 & -2}$ & $\pmat{3 & 1 \\ 6 & 3}$ & $\pmat{0 & -1 \\ 6 & 0}$ \\ \hline
          $A_\ell$ & $\pmat{1 & 0 \\ 0 & 1}$ & $\pmat{1/2 & 0 \\ 0 & 1}$ & $\pmat{1/3 & 0 \\ 0 & 1}$ & $\pmat{1/6 & 0 \\ 0 & 1}$ \\ \hline
          $V_\ell A_\ell$ & $\pmat{1 & 0 \\ 0 & 1}$ & $\pmat{1 & -1 \\ 3 & -2}$ & $\pmat{1 & 1 \\ 2 & 3}$ & $\pmat{0 & -1 \\ 1 & 0}$ \\ \hline
        \end{tabu}
    \end{table}
    Proceeding as in \cite[pp. 10]{spt}, we get
    \begin{equation*}
        \gamma_\infty = V_1A_1, \quad \gamma_{1/3,r} = V_2 A_2 \pmat{1 & r + 1 \\ 0 & 1}, \quad \gamma_{1/2,s} = V_3 A_3 \pmat{1 & s \\ 0 & 1}, \quad \gamma_{0,t} = V_4 A_4 \pmat{1 & t \\ 0 & 1}.
    \end{equation*}
 By (\ref{eq:F-eigen}), $F(V_\ell z) = F(z)$ for $\ell = 1,2$ and $F(V_\ell z) = -F(z)$ for $\ell = 3,6$. Hence, if $\zeta_6 := e(1/6)$ is a primitive sixth root of unity, then
\begin{align*}
        F\vert_0 \gamma_{\infty}(z) = F(z) &= e(-z) - 4 + \sum_{n=1}^{\infty} a(n)e(nz) \\
        F\vert_0 \gamma_{1/3,r}(z) = F\left( \frac{z+r+1}{2} \right) &= \zeta_6^{3-3r} e(-z/2) - 4 + \sum_{n=1}^{\infty} \zeta_6^{3n(r+1)} a(n)e(nz/2) \\
        F\vert_0 \gamma_{1/2,s}(z) = -F\left( \frac{z+s}{3} \right) &= \zeta_6^{3-2s} e(-z/3) + 4 + \sum_{n=1}^{\infty} \zeta_6^{3+2ns} a(n)e(nz/3) \\
        F\vert_0 \gamma_{0,t}(z) = -F\left( \frac{z+t}{6} \right) &= \zeta_6^{3-t} e(-z/6) + 4 + \sum_{n=1}^{\infty} \zeta_6^{3+nt} a(n)e(nz/6).
    \end{align*}

    Meanwhile, a calculation using the definition of $\mathcal{P}_F(z)$ and the group law on the Atkin-Lehner operators shows that
    \begin{align*} 
        \mathcal{P}_F(W_\ell z) = \beta(\ell)\mathcal{P}_F(z),
    \end{align*}
    and hence
    \begin{align*}
        \mathcal{P}_F\vert_0 \gamma_{\infty}(z) = \mathcal{P}_F(z) &= e(-z) + O(1) \\
        \mathcal{P}_F\vert_0 \gamma_{1/3,r}(z) = \mathcal{P}_F\left( \frac{z+r+1}{2} \right) &= \zeta_6^{3-3r} e(-z/2) + O(1) \\
        \mathcal{P}_F\vert_0 \gamma_{1/2,s}(z) = -\mathcal{P}_F\left( \frac{z+s}{3} \right) &= \zeta_6^{3-2s} e(-z/3) + O(1) \\
        \mathcal{P}_F\vert_0 \gamma_{0,t}(z) = -\mathcal{P}_F\left( \frac{z+t}{6} \right) &= \zeta_6^{3-t} e(-z/6) + O(1).
    \end{align*}
    
    From the preceding computations we find that $F$ and $\mathcal{P}_F$ have the same principal parts in the cusps of $\Gamma_0(6)$. Therefore, $F- \mathcal{P}_F$ is a bounded harmonic function on a compact Riemann surface, and hence constant. In particular, we have $F- \mathcal{P}_F = C_F$, where the constant $C_F$ is equal to
    \begin{align*}
        C_F = -4 - b_F(0) + \sum_{n=1}^{\infty} a(n) e(nz) 
        +e(-\bar{z}) - \sum_{n=1}^{\infty} b_F(-n)e(-n\bar{z}) - \sum_{n=1}^{\infty} b_F(n)e(nz).
    \end{align*}
    Take the limit of both sides as $\textrm{Im}(z) \rightarrow \infty$ to get
    \begin{align*}
        C_F &= -4 - b_F(0).
    \end{align*}
    To compute $b_F(0)$, we begin as in \cite[Lemma 3.1]{spt}, utilizing
    \begin{equation*}
        S(-\bar{\ell},0;c) = \mu(c)
    \end{equation*}
    to obtain
    \begin{equation*}
        b_F(0) = 4\pi^2 \sum_{\ell \mid 6} \frac{\beta(\ell)}{\ell} \smashoperator{\sum_{\substack{c > 0 \\ c \equiv 0 \pmod{6/l} \\ (c,\ell)=1}}} \frac{\mu(c)}{c^2}.
    \end{equation*}
    For each $\ell \mid 6$, the rightmost sum then reduces to
    \begin{equation*}
        \sum_{\substack{c > 0 \\ c \equiv 0 \pmod{6/l} \\ (c,\ell)=1}} \frac{\mu(c)}{c^2} = \frac{\ell^2}{36}\sum_{\substack{d=1 \\ (d,\ell)=1}}^{\infty} \frac{\mu(6d/\ell)}{\ell^2} = \frac{1}{\zeta(2)}
        \begin{cases}
            1/24 & \ell = 1 \\
            -1/6 & \ell = 2 \\
            -3/8 & \ell = 3 \\
            3/2 & \ell = 6.
        \end{cases}
    \end{equation*}
    The evaluation $\zeta(2) = \pi^2/6$ then grants
    \begin{equation*}
        b_F(0) = 24\left( \frac{1}{24} - \frac{1}{12} + \frac{1}{8} - \frac{1}{4} \right) = -4.
    \end{equation*}
    It follows that $C_F = 0$ and hence $F=\mathcal{P}_F$. Thus by comparing the Fourier expansion of $F$ and $\mathcal{P}_F$, we obtain
    $a(n)=b_F(n)$ for every $n \geq 1$, $b_F(-1)=1$, and $b_F(-n)=0$ for every $n \geq 2$.
\end{proof}

We conclude this section by giving an effective bound for the Fourier coefficients $a(n)$ for $n \geq 1$.

\begin{lemma} \label{lem:F-fourier-bound}
    For $n \geq 1$,
    \begin{equation*}
        \abs{a(n)} \leq C\exp(4\pi\sqrt{n}),
    \end{equation*}
    where
    \begin{equation*}
        C := 8\sqrt{6}\pi^{3/2} + 16\pi^2\zeta^2(3/2).
    \end{equation*}
\end{lemma}

\begin{proof}
    We utilize the proof of \cite[Lemma 3.1]{spt}, which bounds similar coefficients
    \begin{equation*}
        a'(n) = 2\pi \sum_{\ell \mid 6} \frac{\mu(\ell)}{\sqrt{\ell}} \smashoperator{\sum_{\substack{c > 0 \\ c \equiv 0 \pmod{6/\ell} \\ (c,\ell)=1}}} \frac{S(-\tilde{\ell},n;c)}{c} I_1 \left( \frac{4\pi\sqrt{n}}{c\sqrt{\ell}} \right)
    \end{equation*}
    by $C\sqrt{n}\exp(4\pi\sqrt{n})$ for the given $C$; our result follows then from $\abs{\mu(\ell)} = \abs{\beta(\ell)} = 1$ for all $\ell \mid 6$ and multiplication by $n^{-1/2}$.
\end{proof}

\section{Proof of Theorem \ref{thm:main}} \label{sec:main}

Given a form $Q \in \mathcal{Q}_\Delta$ and corresponding coset representative $\gamma_Q \in \mathbf{C}_6$, let $h_Q \in \{1,2,3,6\}$ be the width of the cusp $\gamma_Q(\infty)$, and let $\zeta_Q$ and $\phi_{n,Q}$ be the sixth roots of unity defined as follows:
\begin{table}[H]
    \centering
    \caption{}
    \label{tbl:zeta}
    \begin{tabu}{|c|c|c|c|c|}
        \hline
         cusp $\gamma_Q(\infty)$ & $\infty \simeq 1/6$ & $1/3$ & $1/2$ & $0 \simeq 1$ \\
         \hline
         $\zeta_Q$ & 1 & $\zeta_6^{3-3r}$ & $\zeta_6^{3-2s}$ & $\zeta_6^{3-t}$ \\
         \hline
         $\phi_{n,Q}$ & $1$ & $\zeta_6^{3n(r+1)}$ & $\zeta_6^{3+2ns}$ & $\zeta_6^{3+nt}$ \\
         \hline
    \end{tabu}
\end{table}
Then from the calculation in Proposition \ref{prop:F-fourier} we can write
\begin{equation} \label{eq:gamma-transform}
    F\vert_0 \gamma_Q(z) = \zeta_Qe(-z/h_Q) - 4\beta(h_Q) + \sum_{n=1}^{\infty} \phi_{n,Q} a(n) e(nz/h_Q).
\end{equation}

Now, recall the Bruinier/Schwagenscheidt formula \cite{BruSch},
\begin{equation} \label{eq:s-to-alpha}
    \alpha(n) = -\frac{1}{\sqrt{\abs{D_n}}} \mathrm{Im}(S(n)).
\end{equation}
We use this to give an effective bound on $S(n)$ and hence obtain our result for $\alpha(n)$. By (\ref{eq:s-decomp}) and (\ref{eq:su-decomp}),
\begin{equation*}
    \begin{split}
        S(n) &= \sum_{\substack{u > 0 \\ u^2 \mid D_n}} \varepsilon(u) S_u(n) \\
        &= \sum_{\substack{u > 0 \\ u^2 \mid D_n}} \varepsilon(u) \sum_{[Q] \in \mathcal{Q}_{D_n/u^2,6,1}^{\mathrm{prim}} / \Gamma_0(6)} F(\tau_Q) \\
        &= \sum_{\substack{u > 0 \\ u^2 \mid D_n}} \varepsilon(u) \sum_{[Q] \in \mathcal{Q}_{D_n/u^2}} F\vert_0 \gamma_Q(\tau_Q)
    \end{split}
\end{equation*}
which, by (\ref{eq:gamma-transform}), yields
\begin{equation*}
    S(n) = \sum_{\substack{u > 0 \\ u^2 \mid D_n}} \varepsilon(u) \sum_{Q \in \mathcal{Q}_{D_n/u^2}} \zeta_Q e(-\tau_Q/h_Q) = \sum_{Q \in \mathcal{Q}_{D_n}} \zeta_Q e(-\tau_Q/h_Q) + E_1(n) + E_2(n)
\end{equation*}
where
\begin{equation*}
    E_1(n) := \sum_{\substack{u > 1 \\ u^2 \mid D_n}} \varepsilon(u) \sum_{Q \in \mathcal{Q}_{D_n/u^2}} \zeta_Q e(-\tau_Q/h_Q)
\end{equation*}
and
\begin{equation*}
    E_2(n) := 4\beta(h_Q)\sum_{\substack{u > 0 \\ u^2 \mid D_n}} \varepsilon(u)h(D_n/u^2) + \sum_{n=1}^{\infty} a(n) \sum_{\substack{u > 0 \\ u^2 \mid D_n}} \varepsilon(u) \phi_{n,Q} e(n\tau_Q/h_Q).
\end{equation*}

To analyze the main term, note that for any $Q = [a_Q, b_Q, c_Q] \in \mathcal{Q}_{D_n/u^2}$, we have
\begin{equation*} 
    a_Qh_Q \equiv 0 \pmod{6}
\end{equation*}
and
\begin{equation} \label{eq:e-tau}
    e(-\tau_Q/h_Q) = \zeta_{2a_Qh_Q}^{b_Q} \exp\left( \frac{\pi\sqrt{\abs{D_n}/u^2}}{a_Qh_Q} \right).
\end{equation}

We consider those forms $Q \in \mathcal{Q}_{D_n}$ with $a_Qh_Q = 6$ and $a_Qh_Q = 12$. We examine Table \ref{tbl:forms}, which contains the value of $c_Q$ for those forms $Q \in \mathcal{Q}_{D_n,6,1}^{\mathrm{prim}} / \Gamma_0(6)$ with $1 \leq a_Q \leq 12$.

\begin{table}[H]
    \centering
    \caption{}
    \label{tbl:forms}
    \begin{tabu}{|c|c c c c c c|} \hline
        $a_Q$\textbackslash $b_Q$ & $\pm 1$ & $\pm 3$ & $\pm 5$ & $\pm 7$ & $\pm 9$ & $\pm 11$ \\ \hline 
        1 & $6n$ & & & & & \\
        2 & $3n$ & & & & & \\
        3 & $2n$ & & & & & \\
        4 & $\frac{3n}{2}$ & $\frac{3n+1}{2}$ & & & & \\
        5 & $\frac{6n}{5}$ & $\frac{6n+2}{5}$ & & & & \\
        6 & $n$ & & $n+1$ & & & \\
        7 & $\frac{6n}{7}$ & $\frac{6n+2}{7}$ & $\frac{6n+6}{7}$ & & & \\
        8 & $\frac{3n}{4}$ & $\frac{3n+1}{4}$ & $\frac{3n+3}{4}$ &  $\frac{3n+6}{4}$ & & \\
        9 & $\frac{2n}{3}$ & & $\frac{2n+2}{3}$ &  $\frac{2n+4}{3}$ & & \\
        10 & $\frac{3n}{5}$ & $\frac{3n+1}{5}$ & &  $\frac{3n+6}{5}$ & $\frac{3n+10}{5}$ & \\
        11 & $\frac{6n}{11}$ & $\frac{6n+2}{11}$ & $\frac{6n+6}{11}$ &  $\frac{6n+12}{11}$ & $\frac{6n+20}{11}$ & \\
        12 & $\frac{n}{2}$ & & $\frac{n+1}{2}$ & $\frac{n+2}{2}$ & & $\frac{n+5}{2}$ \\
        \hline
    \end{tabu}
\end{table}

The forms with $a_Qh_Q = 6$ are then, via \cite[Table 1]{forms},
\begin{equation*}
    Q_1 = [1, 1, 6n], \quad
    Q_2 = [2, 1, 3n], \quad
    Q_3 = [3, 1, 2n], \quad
    Q_4 = [6, 1, n]
\end{equation*}
with coset representatives
\begin{equation*}
    \gamma_{Q_1} = \gamma_{0,1}, \quad
    \gamma_{Q_2} = \gamma_{1/2,-1}, \quad
    \gamma_{Q_3} = \gamma_{1/3,0}, \quad
    \gamma_{Q_4} = \gamma_\infty.
\end{equation*}
Similarly, the forms with $a_Qh_Q = 12$ are
\begin{align*}
    Q_5^0 &= [2,-1,3n] &Q_5^1 &= [2,-1,3n] \\
    Q_6^0 &= [4,1,3n/2] &Q_6^1 &= [4,-3,(3n+1)/2] \\
    Q_7^0 &= [6,-5,n+1] &Q_7^1 &= [6,-5,n+1] \\
    Q_8^0 &= [12,1,n/2] &Q_8^1 &= [12,-11,(n+5)/2]
\end{align*}
with coset representatives
\begin{align*}
    \gamma_{Q_5^0} &= \gamma_{0,0} &\gamma_{Q_5^1} &= \gamma_{0,3} \\
    \gamma_{Q_6^0} &= \gamma_{\frac{1}{2},1} &\gamma_{Q_6^1} &= \gamma_{\frac{1}{2},2} \\
    \gamma_{Q_7^0} &= \gamma_{\frac{1}{3},0} &\gamma_{Q_7^1} &= \gamma_{\frac{1}{3},1} \\
    \gamma_{Q_8^0} &= \gamma_{\infty} &\gamma_{Q_8^1} &= \gamma_{\infty}.
\end{align*}

Thus, for $n \equiv r \pmod{2}$, write
\begin{equation*}
    \sum_{Q \in \mathcal{Q}_{D_n}} \zeta_Q e(-\tau_Q/h_Q) = \sum_{i=1}^{4} \zeta_{Q_i} e(-\tau_{Q_i}/h_{Q_i}) + \sum_{i=5}^{8} \zeta_{Q_i^r} e(-\tau_{Q_i^r}/h_{Q_i^r}) + E_3(n)
\end{equation*}
where
\begin{equation*}
    E_3(n) := \sum_{\substack{Q \in \mathcal{Q}_{D_n} \\ a_Qh_Q \geq 18}} \zeta_Q e(-\tau_Q/h_Q).
\end{equation*}

For $i = 1,2,3,4$, we find via Table \ref{tbl:zeta} the sixth roots of unity
\begin{equation*}
    \zeta_{Q_1} = \zeta_6^2, \quad
    \zeta_{Q_2} = \zeta_6^5, \quad
    \zeta_{Q_3} = \zeta_6^3, \quad
    \zeta_{Q_4} = 1
\end{equation*}
and, for $i = 5,6,7,8$,
\begin{align*}
    \zeta_{Q_5^0} &= \zeta_6^3 &\zeta_{Q_5^1} &= \zeta_6^0 \\
    \zeta_{Q_6^0} &= \zeta_6^1 &\zeta_{Q_6^1} &= \zeta_6^{-1} \\
    \zeta_{Q_7^0} &= \zeta_6^3 &\zeta_{Q_7^1} &= \zeta_6^0 \\
    \zeta_{Q_8^0} &= 1 &\zeta_{Q_8^1} &= 1.
\end{align*}
We then compute via (\ref{eq:e-tau})
\begin{equation*}
     \sum_{i=1}^{4} \zeta_{Q_i} e(-\tau_{Q_i}/h_{Q_i}) = \exp(\pi\sqrt{\abs{D_n}/6}) \sum_{i=1}^{4} \zeta_{Q_i}\zeta_{12}^{b_{Q_i}}
\end{equation*}
where, since $b_{Q_i} = 1$ for $i = 1,2,3,4$,
\begin{equation*}
     \zeta_{12}^{b_{Q_i}}\sum_{i=1}^{4} \zeta_{Q_i} = \zeta_{12}(\zeta_6^3 + \zeta_6^1 + \zeta_6^3 + 1) = 0.
\end{equation*}
Meanwhile, if $n$ is even,
\begin{align*}
     \sum_{i=5}^{8} \zeta_{Q_i^0}\zeta_{24}^{b_{Q_i^0}} = \zeta_{24}^{-1}\zeta_6^3 + \zeta_{24}\zeta_6 + \zeta_{24}^{-5} \zeta_6^3 + \zeta_{24} = i\sqrt{6}
\end{align*}
and, if $n$ is odd,
\begin{align*}
     \sum_{i=5}^{8} \zeta_{Q_i^1}\zeta_{24}^{b_{Q_i^1}} = \zeta_{24}^{-1} + \zeta_{24}^{-3}\zeta_6^{-1} + \zeta_{24}^{-5} + \zeta_{24}^{-11} = -i\sqrt{6}
\end{align*}
so that
\begin{equation*}
    S(n) = (-1)^n i\sqrt{6}\exp(\pi\sqrt{\abs{D_n}}/12) + E_1(n) + E_2(n) + E_3(n).
\end{equation*}
Thus, by (\ref{eq:s-to-alpha}),
\begin{equation*}
    \alpha(n) = (-1)^{n+1} \frac{\sqrt{6}}{\sqrt{24n-1}} e^{l(n)/2} + \mathrm{Im}(E_1(n) + E_2(n) + E_3(n)).
\end{equation*}

We now bound each error term; since $u$ is a unit modulo $12$ and $u > 1$, we have that $u \geq 5$ so that $ua_Qh_Q \geq 30$. Then via (\ref{eq:e-tau}),
\begin{equation*}
    \begin{split}
        \abs{E_1(n)} &\leq \sum_{\substack{u > 1 \\ u^2 \mid D_n}} \sum_{Q \in \mathcal{Q}_{D_n/u^2}} \exp(\pi\sqrt{\abs{D_n}}/a_Qh_Q) \\
        &\leq H(D_n)\exp(\pi\sqrt{\abs{D_n}}/30).
    \end{split}
\end{equation*}
To bound $E_2(n)$, we proceed analogously to \cite[pp. 14--15]{spt} to obtain, via Lemma \ref{lem:F-fourier-bound},
\begin{align*}
    \abs{E_2(n)} &\leq 4H(D_n) + CH(D_n) \sum_{n=1}^{\infty} \exp(4\pi\sqrt{n} - \pi n/2\sqrt{3}) \\
    &\leq C_1H(D_n) 
\end{align*}
where
\begin{equation*}
    C_1 := 4 + C[2.08 \times 10^{20} + 426] < 2.47 \times 10^{23}.
\end{equation*}
Finally,
\begin{equation*}
    \begin{split}
        \abs{E_3(n)} &\leq \sum_{\substack{Q \in \mathcal{Q}_{D_n} \\ a_Qh_Q \geq 18}} \exp(\pi\sqrt{\abs{D_n}}/a_Qh_Q) \\
        &\leq h(D_n)\exp(\pi\sqrt{\abs{D_n}}/18).
    \end{split}
\end{equation*}
Let $E(n) := \mathrm{Im}(E_1(n) + E_2(n) + E_3(n))$; this total error then satisfies
\begin{align*}
    \abs{E(n)} &\leq \abs{E_1(n)} + \abs{E_2(n)} + \abs{E_3(n)} \\
    &\leq H(D_n)\left[ C_1 + \exp(\pi\sqrt{\abs{D_n}}/30) + \exp(\pi\sqrt{\abs{D_n}}/18) \right] \\
    &< (2.48 \times 10^{23})H(D_n)\exp(\pi\sqrt{\abs{D_n}}/18).
\end{align*}
By the class number bound from \cite[pp. 17]{spt}, then,
\begin{equation*}
    \abs{E(n)} < (4.30 \times 10^{23}) 2^{q(n)} \abs{D_n}^2 \exp(\pi\sqrt{\abs{D_n}}/18).
\end{equation*}

\section{Corollaries to Theorem \ref{thm:main}} \label{sec:cor}

We make use of the effective bound on $p(n)$ for all $n \geq 1$ from \cite[Lemma 4.2]{spt}:
\begin{equation} \label{eq:partition-bound}
    p(n) = \frac{2\sqrt{3}}{24n-1} \left(1 - \frac{1}{l(n)} \right) e^{l(n)} + E_p(n)
\end{equation}
where $\abs{E_p(n)} \leq (1313)e^{l(n)/2}$.

\begin{corollary} \label{cor:main}
    For $r = 0,1$ and $n \geq 4$,
    \begin{equation*}
        N(r,2;n) =  M(n) e^{l(n)} + (-1)^r R(n),
    \end{equation*}
    where
    \begin{equation*}
        M(n) := \frac{\sqrt{3}}{24n-1} \left(1 - \frac{1}{l(n)} \right)
    \end{equation*}
    and
    \begin{equation*}
        \abs{R(n)} \leq (8.17 \times 10^{30})e^{l(n)/2}.
    \end{equation*}
\end{corollary}

\begin{proof}
    Utilizing (\ref{eq:partition-bound}) grants, via Theorem \ref{thm:main},
    \begin{align*}
        N(0,2;n) &= \frac{p(n) + \alpha(n)}{2} \\
        &= \frac{\sqrt{3}}{24n-1} \left( 1 - \frac{1}{l(n)} \right) e^{l(n)} + R(n)
    \end{align*}
    and similarly
    \begin{align*}
        N(1,2;n) &= \frac{p(n) - \alpha(n)}{2} \\
        &= \frac{\sqrt{3}}{24n-1} \left( 1 - \frac{1}{l(n)} \right) e^{l(n)} - R(n),
    \end{align*}
    where
    \begin{equation*}
        R(n) := (-1)^{n-1}\frac{\sqrt{6}}{2\sqrt{24n-1}} e^{l(n)/2} + \frac{1}{2}(E_p(n) + E(n)).
    \end{equation*}
    We then have
    \begin{align*}
        \abs{R(n)} &\leq \left(657 + \frac{\sqrt{6}}{2\sqrt{24n-1}}\right) e^{l(n)/2} + (2.15 \times 10^{23})2^{q(n)} \abs{D_n}^2 e^{l(n)/3} \\
        &\leq (8.17 \times 10^{30})e^{l(n)/2}.
    \end{align*}
\end{proof}

\begin{corollary} \label{cor:equi}
    For all $n \geq 4$,
    \begin{equation*}
        \frac{N(r,2;n)}{p(n)} = \frac{1}{2} + (-1)^rR_2(n),
    \end{equation*}
    where
    \begin{equation*}
        \abs{R_2(n)} \leq (1.89 \times 10^{32})e^{-l(n)/3}.
    \end{equation*}
\end{corollary}

\begin{proof}
    Note that
    \begin{equation*}
        \frac{N(r,2;n)}{p(n)} = \frac{1}{2} + (-1)^r\frac{\alpha(n)}{2p(n)}.
    \end{equation*}
    Let $R_2(n) := \alpha(n)/2p(n)$. We utilize a crude lower bound for $p(n)$ for $n \geq 4$
    \begin{equation*}
        p(n) > \frac{\sqrt{3}}{12n}\left( 1 - \frac{1}{\sqrt{n}} \right) e^{l(n)} \geq \frac{\sqrt{3}}{96} e^{l(n)}
    \end{equation*}
    due to Bessenrodt and Ono \cite{ono}, and compute
    \begin{align*}
        \abs{R_2(n)} &\leq \frac{48}{\sqrt{3}} e^{-l(n)} \left(\frac{\sqrt{6}}{\sqrt{24n-1}} e^{l(n)/2} + \abs{E(n)} \right)  \\
        &\leq \frac{48\sqrt{2}}{\sqrt{24n-1}} e^{-l(n)/2} + (1.20 \times 10^{25})2^{q(n)} \abs{D_n}^2 e^{-2l(n)/3} \\
        &\leq (1.89 \times 10^{32})e^{-l(n)/3}.
    \end{align*}
\end{proof}

\section{Proof of Theorem \ref{thm:conj}} \label{sec:conj}

We first require the following lemma:

\begin{lemma} \label{lem:N-bound}
    For $r = 0$ (resp. $r = 1$), we have that 
    \begin{equation*}
        M(n) \left(1 - \frac{1}{\sqrt{n}}\right)  e^{l(n)} < N(r,2;n) < M(n) \left(1 + \frac{1}{\sqrt{n}}\right)  e^{l(n)}
    \end{equation*}
    for all $n \geq 8$ (resp. 7).
\end{lemma}

\begin{proof}
    From Corollary \ref{cor:main}, we have that
    \begin{equation*}
        M(n) e^{l(n)} - \abs{R(n)} < N(r,2;n) < M(n)e^{l(n)} + \abs{R(n)}
    \end{equation*}
    with 
    \begin{equation*}
        \abs{R(n)} \leq (8.17 \times 10^{30}) e^{l(n)/2}.
    \end{equation*}

    We then calculate that, for all $n \geq 4543$, 
    \begin{equation*}
        8.17 \times 10^{30} < \frac{M(n)}{\sqrt{n}} e^{l(n)/2}
    \end{equation*}
    and verify with SageMath \cite{code} and the OEIS \cite{OEIS} the result for $n < 4543$.
\end{proof}

We now proceed with the full proof. Assume $11 \leq a \leq b$ and let $b = Ca$ where $C \geq 1$. By Lemma \ref{lem:N-bound} we have the inequalities
\begin{equation*}
    N(r,2;a)N(r,2;Ca) > M(a)M(Ca) \left( 1 - \frac{1}{\sqrt{a}}\right) \left( 1 - \frac{1}{\sqrt{Ca}}\right)  e^{l(a)+l(Ca)}
\end{equation*}
and
\begin{equation*}
    N(r,2;a+Ca) < M(a+Ca)\left( 1 + \frac{1}{\sqrt{a+Ca}} \right) e^{l(a+Ca)}.
\end{equation*}
Thus, we seek conditions on $a > 1$ such that
\begin{equation*}
    e^{T_a(C)}  > \frac{M(a+Ca)}{M(a)M(Ca)} S_a(C),
\end{equation*}
where
\begin{equation*}
    T_a(C) := l(a) + l(Ca) - l(a+Ca) \text{ and } S_a(C) := \frac{\left(1 + \frac{1}{\sqrt{a+Ca}}\right)}{\left(1 - \frac{1}{\sqrt{a}}\right)\left(1 - \frac{1}{\sqrt{Ca}}\right)}.
\end{equation*}
Taking logarithms in turn grants an equivalent formulation
\begin{equation} \label{eq:log-ineq}
    T_a(C) > \log\left( \frac{M(a+Ca)}{M(a)M(Ca)} \right) + \log S_a(C).
\end{equation}
Furthermore, as functions of $C$, $T_a$ is strictly increasing and $S_a$ strictly decreasing, so that it suffices to show that
\begin{equation*}
    T_a(1) > \log\left( \frac{M(a+Ca)}{M(a)M(Ca)} \right) + \log S_a(1)
\end{equation*}
for all $a \geq 8$, and, with $M(a+Ca)/M(Ca) \leq 1$ for all such $a$, we may show that
\begin{equation} \label{eq:final-ineq}
    T_a(1) > \log S_a(1) - \log M(a).
\end{equation}
Calculation of $T_a(1)$ and $S_a(1)$ shows that (\ref{eq:final-ineq}) holds for $a \geq 18$.

To complete the proof, assume that $11 \leq a \leq 17$. For each such integer $a$, we calculate the real number $C_a$ for which
\begin{equation*}
    T_a(C_a) = \log S_a(C_a) - \log M(a).
\end{equation*}
The values $C_a$ are listed in the table below.
\begin{table}[H]
    \centering
    \caption{}
    \label{tbl:ineq}
    \begin{tabu}{|c|c|c|}
        \hline
         $a$ & $C_a$ & $\max b$ \\ \hline
         $11$ & 2.20\dots & $24$ \\ \hline
         $12$ & 1.86\dots & $22$ \\ \hline
         $13$ & 1.62\dots & $21$ \\ \hline
         $14$ & 1.43\dots & $20$ \\ \hline
         $15$ & 1.27\dots & $19$ \\ \hline
         $16$ & 1.15\dots & $18$ \\ \hline
         $17$ & 1.05\dots & $17$ \\ \hline
    \end{tabu}
\end{table}

By the discussion above, if $b = Ca$ is an integer for which $C > C_a$ holds, then (\ref{eq:log-ineq}) holds, which in turn grants the theorem in these cases. Only finitely many cases remain, namely the pairs integers where $11 \leq a \leq 17$ and $1 \leq b/a \leq C_a$. We compute $N(r,2;a)$, $N(r,2;b)$, and $N(r,2;a + b)$ directly in these cases to complete the proof.

\section{Proof of Theorem \ref{thm:maxN}} \label{sec:maxN}

Let $\lambda^r$ be a partition $(\lambda_1^r,\lambda_2^r,\dots,\lambda_k^r) \in P(n)$ such that $N(r,2;\lambda^r)$ is maximal. These $\lambda^r$ and their corresponding values of $\maxN(r,2;n)$ are recorded in Table \ref{tbl:maxN} for $n \leq 23$, computed using SageMath \cite{code}. Furthermore, let $s_0 = 3$ and $s_1 = 2$, the repeating portions of the conjectured maximal partitions $\lambda^0$ for $n \geq 5$ and $\lambda^1$ for $n \geq 8$ respectively, and $\hat{\lambda}^r$ be the partition obtained by removing all parts of size $s_r$ from $\lambda^r$.

First note that $\lambda^r$ contains no part larger than $23$, since if it did contain some part $i \geq 24$, we could perform the substitution
\begin{equation*}
    (i) \to (\lfloor i/2 \rfloor, \lceil i/2 \rceil)
\end{equation*}
and obtain, by Theorem \ref{thm:main}, a partition $\mu$ such that $N(r,2;\mu) > N(r,2;\lambda^r)$, contradicting the maximality of $N(r,2;\lambda^r)$. Thus, we need only consider parts $i \leq 23$ in $\lambda^r$.

\begin{proposition} \label{prop:mult}
    Let $m_i^r$ be the multiplicity of the part $i$ in $\lambda^r$. If $i \geq 10$, then $m_i^r = 0$. Furthermore, for $r = 0$ and $i \neq s_0$,
    \begin{equation*}
        \begin{cases}
            m_i^0 = 0 & i = 2,4,6,8,9 \\
            m_i^0 \leq 1 & i = 5,7 \\
            m_i^0 \leq 2 & i = 1
        \end{cases}
    \end{equation*}
    and, for $r = 1$ and $i \neq s_1$,
    \begin{equation*}
        \begin{cases}
            m_i^1 = 0 & i = 4,6,8 \\
            m_i^1 \leq 1 & i = 3,5,7,9 \\
            m_i^1 \leq 3 & i = 1
        \end{cases}
    \end{equation*}
\end{proposition}

\begin{proof}
    First note that, for all $i \geq 10$, we may replace $i$ by the representation of $i$ in Table \ref{tbl:maxN} to yield a partition $\mu$ such that $N(r,2;\mu) \geq N(r,2;\lambda^r)$. We then observe the following substitutions for the remaining $i$:
    
    \begin{center}
        \begin{tabu}{l l}
            $r = 0$ & $r = 1$ \\
            $(1,1,1) \to (3)$ & $(1,1,1,1) \to (4)$ \\
            $(2) \to (1,1)$ & \\
            & $(3,3) \to (2,2,2)$ \\
            $(4) \to (1,3)$ & $(4) \to (2,2)$ \\
            $(5,5) \to (3,7)$ & $ (5,5) \to (2,2,2,2,2)$ \\
            $(6) \to (3,3)$ & $(6) \to (2,2,2)$ \\
            $(7,7) \to (3,3,3,5)$ & $(7,7) \to (2,2,2,2,2,2,2)$ \\
            $(8) \to (3,5)$ & $(8) \to (2,2,2,2)$ \\
            $(9) \to (3,3,3)$ & $(9,9) \to (2,2,2,2,2,2,2,2,2)$. \\
        \end{tabu}
    \end{center}
    
    Note in particular that $N(r,2;\mu) = N(r,2;\lambda^r)$ if and only if $r = 1$ and $\mu$ is obtained by the substitutions $(4) \to (2,2)$ or $(6) \to (2,2,2)$. Thus, we may choose a representative of $\lambda^1$ such that $m_4^1 = m_6^1 = 0$, as these substitutions leave $N(1,2;\lambda^1)$ unchanged. This demonstrates the equivalence of partition classes stipulated for $r = 1$ in Conjecture \ref{conj:maxN}.
\end{proof}

\begin{proposition} \label{prop:m0}
    $m_3^1 = m_5^1 = m_7^1 = 0$ unless $\lambda^1 = (3), (5), (7)$, or $(2,5)$.
\end{proposition}

\begin{proof}
    Note that if $m_a^1 \geq 1$ for some $a$, then by Proposition \ref{prop:mult} we know that $a = 1,2,3,5,7,9$. Meanwhile, $m_b^1 \leq 1$ for $b = 3,5,7$. Thus, suppose $m_3^1 = 1$ (resp. $m_5^1 = 1$, $m_7^1 = 1$). Then it can be verified that replacing $(a,3)$ (resp. $(a,5)$, $(a,7)$) with the representation of $a+3$ (resp. $a+5$, $a+7$) in Table \ref{tbl:maxN} will produce a partition $\mu$ with $N(1,2;\mu) \geq N(1,2;\lambda^1)$, with equality only attained for $(2,5) \to (7)$.
\end{proof}

\begin{table}
    \centering
    \caption{}
    \label{tbl:maxN}
    \begin{tabu}{|c c c c c|}
        \hline
        $n$ & $\maxN(0,2;n)$ & $\lambda^0$ & $\maxN(1,2;n)$ & $\lambda^1$ \\ \hline
        1 & 1 & (1) & 0 & (1) \\ \hline
        2 & 1 & (1,1) & 2 & (2) \\ \hline
        3 & 3 & (3) & 0 & (3), (1,2), (1,1,1) \\ \hline
        4 & 3 & (1,3) & 4 & (4), (2,2)  \\ \hline
        5 & 5 & (5) & 2 & (5) \\ \hline
        6 & 9 & (3,3) & 8 & (6), (2,4), (2,2,2) \\ \hline
        7 & 11 & (7) & 4 & (7), (2,5) \\ \hline
        8 & 15 & (3,5) & 16 & (2,2,2,2) \\ \hline
        9 & 27 & (3,3,3) & 12 & (9) \\ \hline
        10 & 33 & (3,7) & 32 & (2,2,2,2,2) \\ \hline
        11 & 45 & (3,3,5) & 24 & (2,9) \\ \hline
        12 & 81 & (3,3,3,3) & 64 & (2,2,2,2,2,2) \\ \hline
        13 & 99 & (3,3,7) & 48 & (2,2,9) \\ \hline
        14 & 135 & (3,3,3,5) & 128 & (2,2,2,2,2,2,2) \\ \hline
        15 & 243 & (3,3,3,3,3) & 96 & (2,2,2,9) \\ \hline
        16 & 297 & (3,3,3,7) & 256 & (2,2,2,2,2,2,2,2) \\ \hline
        17 & 405 & (3,3,3,3,5) & 192 & (2,2,2,2,2,9) \\ \hline
        18 & 729 & (3,3,3,3,3,3) & 512 & (2,2,2,2,2,2,2,2,2,2) \\ \hline
        19 & 891 & (3,3,3,3,7) & 384 & (2,2,2,2,2,9) \\ \hline
        20 & 1215 & (3,3,3,3,3,5) & 1024 & (2,2,2,2,2,2,2,2,2,2,2) \\ \hline
        21 & 2187 & (3,3,3,3,3,3,3) & 768 & (2,2,2,2,2,2,9) \\ \hline
        22 & 2673 & (3,3,3,3,3,7) & 2048 & (2,2,2,2,2,2,2,2,2,2,2,2) \\ \hline
        23 & 3645 & (3,3,3,3,3,3,5) & 1536 & (2,2,2,2,2,2,2,9) \\ \hline
    \end{tabu}
\end{table}

\begin{proposition} \label{prop:m1}
   There exist no distinct $a,b \neq s_r$ such that $m_a^r = m_b^r = 1$.
\end{proposition}

\begin{proof}
    By Proposition \ref{prop:mult}, we know that for $r = 0$ (resp. $r = 1$), $m_a^r = m_b^r = 1$ implies $a,b \in \{1,5,7\}$ (resp. $a,b \in \{1,9\}$ via Proposition \ref{prop:m0}). It can then be verified that replacing $a$ and $b$ with the representation of $a + b$ in Table \ref{tbl:maxN} will yield a partition $\mu$ with $N(r,2;\mu) > N(r,2;\lambda^r)$.
\end{proof}

\begin{proposition} \label{prop:1m}
    $m_1^r = 0$ unless
    \begin{align*}
        &\lambda^0 = (1), (1,1), (1,3) \\
        &\lambda^1 = (1), (1,2), (1,1,1).
    \end{align*}
\end{proposition}

\begin{proof}
    Suppose that $m_1^r \geq 1$. By Proposition \ref{prop:m1}, we know that $\hat{\lambda}^0 = (1), (1,1)$ (resp. $\hat{\lambda}^1 = (1), (1,1), (1,1,1)$). Now we add back in the parts of size $s_r$, and observe the following substitutions which yield partitions $\mu$ such that $N(r,2;\mu) > N(r,2;\lambda^r)$:
    \begin{center}
        \begin{tabu}{l l}
            $r = 0$ & $r = 1$ \\
            $(1,3,3) \to (7)$ & $(1,2,2) \to (5)$ \\
            $(1,1,3) \to (5)$ & $(1,1) \to (2)$ \\
            & $(1,1,1,2) \to (5)$. \\
        \end{tabu}
    \end{center}
\end{proof}

We now complete the proof of Theorem \ref{thm:maxN}. For $r = 0$ (resp. $r = 1$), suppose $\lambda^r \in P(n)$ for $n \geq 5$ (resp. $8$). By Proposition \ref{prop:m1} and Proposition \ref{prop:1m}, we know that $\hat{\lambda}^0 = (5), (7)$ (resp. $\hat{\lambda}^1 = (9))$. These partitions cover all the residue classes modulo $s_r$ except for $n \equiv 0 \pmod{s_r}$ exactly once. For such $n$, appending parts of size $s_r$ to these partitions covers each $n$ exactly once and yields the partitions stipulated in Theorem \ref{thm:maxN}. If $n \equiv 0 \pmod{s_r}$, we can deduce that $\lambda^r = (s_r,s_r,\dots,s_r)$ as stated in Theorem \ref{thm:maxN}; the values of $\maxN(r,2;n)$ then follow.

\bibliography{references}
\bibliographystyle{amsplain}

\end{document}